\documentclass{article}
\usepackage[T1]{fontenc}
\usepackage[utf8]{inputenc}
\usepackage[english]{babel}
\usepackage{amsmath,enumerate}
\usepackage{amssymb}
\usepackage{amscd}
\usepackage{vmargin}
\usepackage{url}
\usepackage{stmaryrd}
\usepackage{tikz,pgf}
\usepackage{subfig,color}

\title{Planar digraphs without large acyclic sets}

\author{Kolja Knauer\footnote{\noindent
Aix-Marseille Université, CNRS, LIF UMR 7279, 13288, Marseille, France}
\and Petru Valicov\footnotemark[1]
\and Paul S. Wenger\footnote{\noindent School of Mathematical Sciences, Rochester Institute of Technology, Rochester, NY, United States of America}}

\begin{document}
\maketitle
\thispagestyle{empty}

\begin{abstract}
Given a directed graph, an acyclic set is a set of vertices inducing a directed subgraph with no directed cycle. In this note we show that for all integers $n\geq g\geq 3$, there exist oriented planar graphs of order $n$ and digirth $g$ for which the size of the maximum acyclic set is at most $\lceil \frac{n(g-2)+1}{g-1} \rceil$. When $g=3$ this result disproves a conjecture of Harutyunyan and shows that a question of Albertson is best possible.
\end{abstract}

\newenvironment{proof}{\par \noindent \textbf{Proof}}{\hfill$\Box$ \bigskip}

\newtheorem{corollary}{Corollary}
\newtheorem{definition}{Definition}
\newtheorem{question}{Question}
\newtheorem{problem}{Problem}
\newtheorem{proposition}{Proposition}
\newtheorem{theorem}{Theorem}
\newtheorem{lemma}{Lemma}
\newtheorem{conjecture}{Conjecture}
\newtheorem{sketch}{Sketch of proof}
\newtheorem{observation}{Observation}
\newtheorem{remark}{Remark}
\newtheorem{claim}{Claim}
\newtheorem{example}{Example}

\newcommand{\comment}[1]{\textcolor{red}{#1}}

\section{Introduction}

An \emph{oriented graph} is a digraph $D$ without loops and multiple arcs. An \emph{acyclic set} in $D$ is a set of vertices which induces a directed subgraph without directed cycles. The complement of an acyclic set of $D$ is a \emph{feedback vertex set} of $D$. A question of Albertson, which was the problem of the month on Mohar's web page~\cite{B} and was listed as a "Research Experience for Graduate Students" by West~\cite{West}, asks whether every oriented planar graph on $n$ vertices has an acyclic set of size at least $\frac{n}{2}$. There are three independent strengthenings of this question in the literature. In the following, we discuss them briefly.

\begin{conjecture}[Harutyunyan~\cite{H11}~\cite{HM14}]
\label{conj:Harutyunyan}
Every oriented planar graph of order $n$ has an acyclic set of size at least $\frac{3n}{5}$.
\end{conjecture}

The \emph{digirth} of a directed graph is the length of a smallest directed cycle. Golowich and Rolnick~\cite{GR14} showed that a oriented planar graph of digirth $g$ has an acyclic set of size at least $\max(\frac{n(g-3)+6}{g},\frac{n(2g-3)+6}{3g})$, in particular proving Conjecture~\ref{conj:Harutyunyan} for oriented planar graphs of digirth~$8$.

A lower bound of $\frac{n}{2}$ for the size of an acyclic set in an oriented planar graph would immediately follow from any of the following two conjectures.

\begin{conjecture}[Neumann-Lara~\cite{NL85}]
\label{conj:Neumann-Lara}
Every oriented planar graph can be vertex-partitioned into two acyclic sets.
\end{conjecture}

Harutyunyan and Mohar~\cite{HM14} proved Conjecture~\ref{conj:Neumann-Lara} for oriented planar graphs of digirth $5$.
The undirected analogue of Conjecture~\ref{conj:Neumann-Lara} is false. Indeed, it is equivalent to a conjecture of Tait~\cite{Tait}, saying that every $3$-connected planar cubic graph has a Hamiltonian cycle, which was disproved by Tutte~\cite{Tutte}. However, the following question remains open:

\begin{conjecture}[Albertson and Berman~\cite{AB79}]
\label{conj:Albertson-Berman}
Every simple undirected planar graph of order $n$ has an induced forest of order at least $\frac{n}{2}$.
\end{conjecture}

There are many graphs showing that Conjecture~\ref{conj:Albertson-Berman}, if true, is best-possible, e.g., $K_4$ and the octahedron. The best-known lower bound for the order of a largest induced forest in a planar graph is $\frac{2n}{5}$ and follows from Borodin's result on acyclic vertex-coloring of undirected planar graphs~\cite{B76}.

We summarize the discussion in Table~\ref{table}:

\begin{table}[!ht]
\begin{center}

 \begin{tabular}{l|c|c|c|c}
  digirth $g$ & $3$ & $4$ & $5$ & $\geq 6$\\
  \hline
  acyclic set & $\frac{2n}{5}$~\cite{B76} &  $\frac{5n+6}{12}$~\cite{GR14} & $\frac{n}{2}$~\cite{HM14} & $\frac{n(g-3)+6}{g}$~\cite{GR14}
 \end{tabular}

\end{center}
\caption{Lower bounds for acyclic sets in oriented planar graphs.}\label{table}
\end{table}

\section{The construction}

In this section, we construct oriented planar graphs of a given digirth with no large acyclic sets.
The most important case of this result is the one when the digirth is 3. Here our result implies that, if true, for odd $n$ the lower bound of $\frac{n}{2}$ in Albertson's question is best possible, while it might be improved by at most $1$ in the even case, see Problem~\ref{prob}. In particular, this disproves Conjecture~\ref{conj:Harutyunyan}. 

\begin{theorem}\label{thm:ub}
 For all integers $n\geq g\geq 3$, there exists an $n$-vertex oriented planar graph with digirth $g$ in which the maximum acyclic set has size $n-\lfloor\frac{n-1}{g-1}\rfloor = \lceil\frac{n(g-2)+1}{g-1}\rceil$.
\end{theorem}
\begin{proof}

Let $g \geq 3$ be fixed. We inductively show that for any $f \geq 1$ there exists a oriented planar graph $D_f$ such that: $D_f$ has digirth $g$, order $f(g-1)+1$, and a minimum vertex feedback set of size $f$. Moreover, $D_f$ has two vertices $x$ and $y$ on a common face $F$, which do not simultaneously appear in any minimum feedback vertex set. For $D_1$ we take a directed cycle of order $g$. This clearly satisfies all conditions. 
 If $f>1$, take the plane digraph $D_{f-1}$, add a directed path $s_1, \ldots, s_{g-1}$ into its face $F$ and add arcs $xs_1$, $ys_1$, $s_{g-1}x$, and $s_{g-1}y$. See Figure~\ref{fig:dessin}.

\begin{figure}[!ht]
  \centering
  \includegraphics{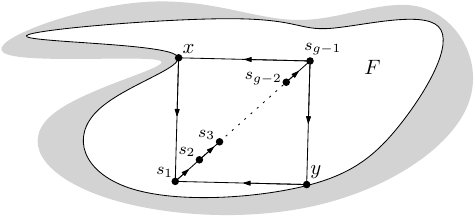}
  \caption{The construction in Theorem~\ref{thm:ub}}
  \label{fig:dessin}
\end{figure} 
 
 Since in $D_{f-1}$ no minimum feedback vertex set uses both $x$ and $y$, in order to hit the two newly created directed cycles, an additional vertex $z$ is needed. Thus, $D_f$ has a minimum feedback vertex set of size at least $f$. Now, choosing $z\in\{s_1, \ldots, s_{g-1}\}$, indeed gives a feedback vertex set of size $f$. Note that in $D_f$ there is no minimum feedback vertex containing two vertices from $\{s_1, \ldots, s_{g-1}\}$ and each such pair of vertices lies on a common face. It only remains to check the order of $D_f$. Since we added a total of $g-1$ new vertices, $D_f$ has $f(g-1)+1=:n$ vertices, i.e., $f=\frac{n-1}{g-1}$. Thus, the largest acyclic set in $D_f$ is of size $n-\frac{n-1}{g-1}$. Therefore, this construction proves the theorem for the case when $g-1$ divides $n-1$.
 
If $n-1$ is not divisible by $g-1$ then we do our construction for the largest integer $n'\leq n$ such that $n'-1$ is divisible by $g-1$, that is $n'=(g-1)\lfloor\frac{n-1}{g-1}\rfloor + 1$. We add $n-n'$ independent vertices to this graph. Now, the largest acyclic set of the obtained graph is of size $n'-\frac{n'-1}{g-1}+(n-n')=n-\lfloor\frac{n-1}{g-1}\rfloor$.
\end{proof}

By Theorem~\ref{thm:ub}, for even $n$, there exist $n$-vertex oriented planar graphs in which every acyclic set has size at most $\frac{n}{2}+1$. A computer check, using tools from Sage~\cite{Sage}, shows that there are ten planar triangulations with $n$ vertices ($n$ is even and $n\leq 10$) that are tight examples for Conjecture~\ref{conj:Albertson-Berman}. Furthermore, for all orientations of these examples, the largest directed acyclic set is of size at least $\frac{n}{2}+1$. We wonder if the following is true:

\begin{problem}\label{prob}
If a largest induced forest in a simple undirected planar graph $G$ on $n$ vertices is of size $a\leq\frac{n}{2}$, then for every orientation $D$ of $G$ there is an acyclic set of size at least $a+1$. 
\end{problem}

\subsubsection*{Acknowledgements} The authors thank the anonymous referees for helpful remarks. In particular, the reference to West's web page~\cite{West}, where a construction for the digirth $3$ case is claimed, was pointed out by one referee. This led to the collaboration of the third with the first two authors. The first author was also supported by PEPS grant EROS.

\end{document}